\documentclass[a4paper, 11pt]{amsart}
\input xy   
\xyoption{all}
\swapnumbers
\usepackage{amssymb}
\usepackage{amsthm}
\usepackage{enumerate}
\usepackage[active]{srcltx}
\usepackage{color}
\numberwithin{equation}{section}

\theoremstyle{plain}
  \begingroup
        \newtheorem{theorem}[equation]{Theorem}
        
        \newtheorem{proposition}[equation]{Proposition}
        \newtheorem{corollary}[equation]{Corollary}

	    \newtheorem{definition}[equation]{Definition}

  \endgroup

\theoremstyle{definition}
  \begingroup
        \newtheorem{remark}[equation]{Remark}
        
        \newtheorem{comment}[equation]{Comment}

  \endgroup

\newcommand{\mr}[1]{\buildrel {#1} \over \longrightarrow}

\newcommand{\rimply}{\Rightarrow}

\newcommand{\vs}{\vspace}
\newcommand{\hs}{\hspace}

\newcommand{\mmr}[1]{\buildrel {#1} \over \hookrightarrow}

\newcommand{\cqd}{\hfill$\Box$}

\begin{document}


\title{Some remarks on infinitesimals in MV-algebras}

\author{Eduardo J. Dubuc and Jorge C. Zilber}


\maketitle

\begin{abstract}
Replacing $\{0\}$ by the whole ideal of infinitesimals yields a weaker notion of \emph{archimedean element} that we call \emph{quasiarchimedean}. It is known that semisimple MV-algebras with compact maximal spectrum (in the co-Zarisky topology) are exactly the hyperarchimedean algebras. We characterise all the algebras with compact maximal spectrum as being \emph{quasihyperarchimedean} \mbox{MV-algebras,} which in a sense are non semisimple hyperarchimedean algebras. We develop some basic facts in the theory of MV-algebras along the lines of algebraic geometry, where infinitesimals play the role of nilpotent elements, and prove a MV-algebra version of Hilbert's Nullstellensatz. Finally we consider the relations (some inedited) between several elementary classes of MV-algebras in terms of the ideals that characterise them, and present elementary (first order with denumerable disjunctions) proofs in place of the \mbox{set-theoretical} usually found in the literature.
\end{abstract}

\tableofcontents



\section{Background on MV-algebras of continuos functions}

%


Given an MV-algebra $A$, $X_A \subset [0,1]^{A}$ will denote the set of morphisms of $A$ into the MV-algebra $[0,\; 1]$. $X_A$ becomes a compact Hausdorff space with the topology inherited from the product space. It is immediate to see that a base for the product topology is given by the subsets 
\mbox{$W_a = \{\chi \,| \, \chi(a) > 0\} \subset X_A$,} that is, the \emph{Zariski} topology. On the other hand, this set of morphisms also inherited a topology as a subspace of the prime spectrum $Z_A$ via the map $\chi \mapsto Kernel(\chi) \in M_A \subset Z_A$, where $M_A$ denotes the maximal spectrum (see \cite{DP1}). We will denote this space by 
$X^c_a \cong M_A$, its topology is the \emph{coZariski} topology with a base of open sets given by the complements of the subsets $W_a$, that we denote 
$W^c_a = \{\chi \,| \, \chi(a) = 0\}$,  The $W^c_a$ are also closed in $X^c_a$ (\cite[4.2]{DP1}), which shows that the coZariski topology is finer than the Zariski topology. Given any MV-algebra $A$, 

\vspace{1ex}

{
\refstepcounter{equation}  \noindent (\theequation)  \label{DP_1}
\emph
    {   
\cite[4.16]{DP1}:  
Equality is a continous bijection $X^c_A \mr{=} X_A$. 
    }                   
}


Each element $a \in A $ determines a continuous function 
$X_A \mr{\widehat{a}} [0,\;1]$,
$\widehat{a}(\chi) = \chi (a), \; for \; \chi \in X_A$. This determines a morphism 
$A \mr{ } Cont(X_A)$, with image denoted 
$\widehat{A} \subset Cont(X_A)$. 
Note that $W^c_a = \widehat{a}^{-1}(0)$. 

\vspace{1ex}

Consider the MV-algebra $Cont(X)$ of 
$[0,\,1]$-valued continuous functions on a topological space $X$, and 
let $A \subset Cont(X)$ be a subalgebra. Recall that $A$ is said to be \emph{separating} iff for any two distinct point $x$ and $y$, there is $f \in A$ such that $f(x) = 0$ and $f(y) > 0$. Each $x \in X$ determines a morphism 
$A \mr{\widehat{x}} [0,\;1]$ defined by $\widehat{x}(a) = a(x)$. This determines a continuous function $X \mr{\varepsilon} X_A$. If $X$ is compact Hausdorff and $A$ is separating, we have:

\vspace{1ex}

 {
\refstepcounter{equation}  \noindent (\theequation)  \label{CDM_1}
\emph
    {
     \cite[4.1]{CDM}: \hspace{2ex}  
 The  map $\varepsilon: \, X \mr{\cong} X_A$ is a homeomorphism.                    
    }
}

\vspace{1ex}

Given an ideal $I \subset A$, we denote by $V(I)$ the locus of roots of the functions $f \in I$, 
$
\; V(I) = \{ x \in X \; | \; f(x) = 0 \; \forall \, f \in I\}, \hspace{1ex} (V(I) \subset X \;\; \text{is a closed subset}).
$

\vspace{1ex}

Given a closed subset $S \subset X$, we denote by $J(S)$ the set of all functions null on $S$,  
$
\; J(S) = \{f \in A \; | \; f(x) = 0 \; \forall \, x \in S\},  \hspace{1ex} (J(S) \subset A \;\; \text{ is an ideal}).
$
 
\vspace{1ex}
 
It is immediate to check that the maps $S \mapsto J(S)$ and $I \mapsto V(I)$ are order reversing and that $I \subset J(V(I))$ and $S \subset V(J(S))$. 
 
 If $X$ is compact Hausdorff and $A$ is separating, we have:

\vspace{1ex}

{
\refstepcounter{equation}  \noindent (\theequation)  \label{eca1}
\emph
    {
     \cite[3.4.2]{COM}: $\;\;\; 
 V(J) \neq \emptyset \; for \; each \; proper \; ideal \; J.$                   
    }
}

\vspace{1ex}

{
\refstepcounter{equation}  \noindent (\theequation)  \label{eca2}
\emph
    {
     \cite[3.4.3]{COM}:  $\; \;\;
     S = V(J(S)) \; for \; each \; closed \; subset \; S.$                  
    }
}

\vspace{1ex}

{
\refstepcounter{equation}  \noindent (\theequation)  \label{eca3}
\emph
    {
    It follows that for $f \in A$, $f \in J(S) \;\;  \iff  \;\; f|_S = 0$, thus:
      \begin{center}  
                      $A/(J(S) \cap A ) \cong A|_S$
      \end{center}
    }
}
Recall:

{
\refstepcounter{equation}  \noindent (\theequation)  \label{farchi}
\emph
    {
\cite[4.5]{CDM}: Given any compact space $X$ and any $f \in Cont(X)$, we have  
         \begin{center}                    
     $f$ is archimedean  
      $\iff V(\langle f \rangle) = f^{-1}(0) \subset X$ is open.   
      \end{center}
    }
}

\vspace{1ex}

\noindent where $\langle f \rangle \subset Cont(X)$ is the ideal generated by $f$. Note that under the homeomorphism (\ref{CDM_1}) 
$V(\langle f \rangle) \cong W^c_f \subset X_{Cont(X)}$.  

\vspace{1ex}

Given any families of ideals $\{I_\ell\}_{I_\ell \in L}$ and of closed subsets $\{S_\ell\}_{\ell \in L}$, from the universal property which defines supremum and infimum it immediately follows:
\begin{equation} \label{gn1}
 {\textstyle \bigvee_{\ell \in L} V(I_\ell) 
                     \subset V(\bigwedge_{\ell \in L} I_\ell)}, 
   \hs{5ex} {\textstyle \bigvee_{\ell \in L} J(S_\ell) 
                     \subset J(\bigwedge_{\ell \in L} S_\ell)}. 
\end{equation}
\begin{equation} \label{gn2}
{\textstyle \bigwedge_{\ell \in L} V(I_\ell) 
                     =  V(\bigvee_{\ell \in L} I_\ell)}, 
   \hs{5ex} {\textstyle \bigwedge_{\ell \in L} J(S_\ell)
                     =  J(\bigvee_{\ell \in L} S_\ell)}.
\end{equation} 
(the infima here are the set theoretical intersection, but the suprema not). 

\vspace{1ex}

{\sc Free MV-algebras}
   
For each set $N$, we denote by $F[N]$ the free MV-algebra on \mbox{$N$-generators.} $F[N]$ is the MV-algebra of terms $f$ in variables 
$\{x_i\}_{i \in N}$. 

Note that with the hindsight of category theory free algebras should be considered up to isomorphisms. In this way we associate free algebras to sets, not to cardinals. Any two bijective sets determine isomorphic algebras.

By Chang's completeness theorem $F[N]$ can be considered to be the \mbox{MV-algebra} of $[0,\, 1]$-valued term functions on the compact space 
$[0,\;1]^N$. Term functions are continuous and it is not difficult to prove they are separating.

\vs{1ex}

{
\refstepcounter{equation} \noindent (\theequation)  \label{chang}
\emph
    {
     \cite[3.4.6]{COM}: $\;F[N]$ can be considered to be the separating 
     subalgebra of term functions
     $F[N] \subset Cont(X)$, for the compact space $X = [0,\, 1]^N$.     
    }
}

\vs{1ex}

%
%
%

%
%

Given a MV-algebra with a presentation $A = F[N]/I$, $a_i = [x_i]$, the universal properties of the free algebra and the quotient algebra say (in turn) that the restriction along $N \mmr{i} F[N]$ determines a continuous bijection
 $i^*: X_{F[N]} \mr{\cong} [0,\;1]^N$  that restricts to a bijection $i^*: X_A \mr{\cong} V(I)$. Since the spaces are compact Hausdorff they are homeomorphisms.

\vs{1ex}

{
\refstepcounter{equation} \noindent (\theequation)  \label{chang}
\emph
    {
    If $A = F[N]/I$, then the restriction along $N \mmr{i} F[N]$ determines a homeomorphism  
    $i^*: X_A \mr{\cong} V(I) \subset  [0,\;1]^N$, 
    $i^*(\chi) = (\chi(a_i))_{i \in N}$.   
    }
}

\vs{1ex}

\section{Quasihyperarchimedean algebras} \label{quasihyperarchimedean}

Recall that an element $a$ in an MV-algebra $A$ is said to be \emph{infinitesimal} if for each integer $n \geq 0$, $ na\leq \neg a$, equivalently, iff \mbox{$na \ominus \neg a = na \odot a = 0$}.\footnote{caution: Contrary with common usage, we consider $0$ to be infinitesimal, as in algebraic geometry $0$ is considered to be nilpotent.}
\begin{remark}(\cite[3.6.3]{COM}) \label{arqinf=0}
For any infinitesimal element $a > 0$, the sequence 
$(0 \leq a \leq 2a \leq 3a \leq \;\ldots \;  \leq  na \leq  \ldots \;\;)$ is strictly increasing. \cqd
\end{remark}
Recall that an element $a$ in an MV-algebra $A$ is said to be $archimedean$ if there is an integer $n \geq 0$, such that $ (n + 1) a \ominus na = 0$, equivalently, iff the sequence $(a \leq 2a \leq 3a \leq \;\ldots \;  \leq  na \leq  \ldots \;\;)$ is stationary.

\vspace{1ex}

\emph{Note that it follows that the only archimedean infinitesimal is $0$.}

\vspace{1ex}

For any ideal $I$ it follows by an easy induction:

\begin{remark} \label{nimpliesmk}
Given $x \in A$ and an integer $n \geq 1$,   
     if  \mbox{$(n+1)x \ominus nx \in I$}, 
     then \mbox{$\forall \, m > k \geq n, \; mx \ominus kx \in I$.}
 \cqd
\end{remark}

\begin{definition} \label{quasiarchimedean}
An element $a$ in an MV-algebra $A$ is said to be \emph{quasiarchimedean} if there is an integer $n \geq 0$, such that $ (n + 1) a \ominus na$ is infinitesimal. A MV-algebra is \emph{quasihyperarchimedean} if every element is quasiarchimedean.
\end{definition}
Clearly archimedean elements are quasiarchimedean, and hyperarchimedean algebras are quasihyperarchimedean.

\begin{proposition} \label{archiquasiarchi}
$a \in A$ is quasiarchimedean $\; \iff \;$  $\widehat{a} \in \widehat{A}$ is archimedean.
\end{proposition}
\begin{proof}
One implication is clear since any morphism preserves quasiarchimedean elements, and the only infinitesimal in $\widehat{A}$ is $0$. For the other implication, take $n \geq 0$ such that 
$(n+1)\widehat{a} \ominus n\widehat{a} = 0$. Then for all $\chi \in X_A$,  
$0 = (n+1)\widehat{a}(\chi) \ominus n\widehat{a}(\chi) =
(n+1)\chi(a) \ominus n \chi(a) = \chi((n+1)a \ominus na)$. Thus  
$(n+1)a \ominus na \in Rad(A) = \sqrt{A}$, that is, it is infinitesimal. 
\end{proof}

From (\ref{farchi}) and Proposition \ref{archiquasiarchi} it immediately follows:
\begin{proposition}
$a \in A$ is quasiarchimedean $\iff$ 
$W^c_a \subset X_A$ is open.

\vspace{1ex}

\noindent  (this corrects the asymmetry  in propositions 5.4 and 5.6 of \cite{DP1}). \cqd
\end{proposition}
We establish now a characterisation of quasihyperarchimedean \mbox{MV-algebras} as those algebras with a compact maximal spectrum. The reader should note that the maximal spectrum $M_A \subset Z_A$ is in this case a compact Hausdorff non closed subspace of the compact prime spectrum.

\begin{proposition}
The following conditions in a MV-algebra are equivalent:

(1) $A$ is quasihyperarchimedean.

(2) For all $a \in A$, $W^c_a \subset X_A$ is open (thus clopen).

(3) The map $X_A \mr{} X^c_A \cong M_A$ is continuous (thus a homeomorphism).

(4) The maximal spectrum  $X^c_A \cong M_A$ is compact.
\end{proposition}
\begin{proof} Clearly $(1) \iff (2)$, and $(3) \implies (4)$. $(4) \implies (3)$ because then (see \ref{DP_1}) $X^c_A \mr{} X_A$ is a continuous bijection between compact Hausdorff spaces. Finally, $(2) \iff (3)$ because the sets $W^c_a$ are an open \mbox{base of  $X^c_A$.}
\end{proof}

\section{The MV-Nullstellensatz.} \label{infandradical}

In this section we develop some basic lines of algebraic geometry in the context of MV-algebras (reference is \cite{HA}). As nilpotent elements are considered "infinitesimal" in algebraic geometry, here its role is played by the MV-algebra concept of, properly called, infinitesimal elements. 

\vspace{1ex}

We start by recalling a first-order (with denumerable disjunctions) characterisation of maximal ideals, which is a key result in the theory of \mbox{MV-algebras} (\cite[1.2.2]{COM}). For any MV-algebra $A$ and ideal $I \subset A$, 
\begin{equation} \label{basic}
I \; is \; maximal \;\; \iff \;\; \forall x \in A \; (x \notin I \; \iff  
\exists \, n \geq 1\; | \; \neg nx \; \in I). 
\end{equation}

\vspace{1ex}

The intersection of all maximal ideals of a MV-algebra $A$ is an ideal called the \emph{radical} of $A$, and denoted $Rad(A)$. In the light of this, we define:
\begin{definition} \label{radicalideals}
Given an ideal $I \subset A$, the intersection of all maximal ideals $M \supset I$ containing $I$ is an ideal that we call the \emph{radical} of $I$, denoted $Rad(I)$. $I$ is called a \emph{radical} ideal if $I = Rad(I)$.
\end{definition}
\begin{remark} \label{radicalprime}
Recall that if $I$ is a prime ideal, then it is contained in a unique maximal ideal \cite[1.2.12]{COM}. It follows that  $Rad(I)$ is a maximal ideal. 
\end{remark}
\begin{proposition} \label{-Rad=Rad-}
Let $A \mr{\varphi} B$ be a surjective morphism of MV-algebras, and 
$I \subset B$ any ideal of $B$. Then:
$$
\varphi^{-1}Rad(I) = Rad(\varphi^{-1}I).
$$
\end{proposition}
\begin{proof}
It follows once we observe that for any pair of ideals $M,\, I$ in $B$, $M \supset I$ iff $\varphi^{-1}M \supset \varphi^{-1}I$, and $M$ is maximal iff $\varphi^{-1}M$ is maximal (the second equivalence follows easily from (\ref{basic}) above).
\end{proof}
\begin{proposition} \label{Rad=JV}
Let $X$ be a compact space, $A \subset Cont(X)$ a separating subalgebra, and $I \subset A$ any ideal. Then, $Rad(I) = J(V(I))$. Thus, $I$ is a radical ideal iff $I = J(V(I))$.
\end{proposition}
\begin{proof} 
Once we observe that for any point $x \in X$, $x \in V(I)$ iff $I \subset J(\{x\})$, the proof follows immediately from (\ref{eca3}) above.
\end{proof}
From this proposition and (\ref{eca2}) above it follows: 
\begin{proposition}
Given a compact space $X$ and a separating subalgebra $A \subset Cont(X)$, the correspondence given by $J$ and $V$ establishes a bijection between the closed subsets of $X$ and the radical ideals of $A$.  \cqd
\end{proposition}


\vspace{1ex}

We call the set of infinitesimals (see section \ref{quasihyperarchimedean}) the \emph{infradical} of $A$, and denote it by  
$\sqrt{A}$. It is well known that $\sqrt{A} = Rad(A)$ \cite[3.6.4]{COM}, but we will not need this here, neither that the set $\sqrt{A}$ is an ideal. All this will be a particular case of our more general Theorem \ref{nulle}. Note that $\sqrt{[0,\,1]} = \{0\}$.

\vspace{1ex}

The following definition was communicated to us by R. Cignoli \cite{CI}, compare with \cite[page 48]{HA}. 
\begin{definition}  \label{I-infinitesimals}
Let $I \subset A$ be an ideal of a MV-algebra $A$. An element $a$ in $A$ is said to be $I$-infinitesimal iff $\, na \ominus \neg a \in I$ for each integer 
$ n \geq 0$. Clearly an element $a$ is $I$-infinitesimal iff $\rho(a)$  is infinitesimal in the quotient algebra $A \mr{\rho} A/I$.
\end{definition} 



We call this set the \emph{infradical} of 
$I$, and we denote it by $\sqrt{I}$. Since \mbox{$na \ominus \neg a = na \odot a \leq a$,} it follows $I \subset \sqrt{I}$.
It is immediate to check the following two propositions.
\begin{proposition} \label{radp1}
Let $\{I_\ell \}_{\ell \in L}$ be any family of ideals. Then
$$ 
{\textstyle \sqrt{\,\bigcap_{\ell \in L} I_\ell \,} \;=\; \bigcap_{\ell \in L} \sqrt{I_\ell}.}
$$

\vspace{-4ex} 

\cqd \end{proposition}
\begin{proposition} \label{-Irad=Irad-}
Let $A \mr{\varphi} B$ be any morphism of MV-algebras, and 
$I \subset B$ any ideal of $B$. Then:
$$
\varphi^{-1}\sqrt{I} = \sqrt{\varphi^{-1}I}.
$$

\vspace{-4ex} 

\cqd \end{proposition}
%
\begin{proposition} \label{prenull}
 Let $X$ be any topological space, and \mbox{$A \subset Cont(X)$} any subalgebra (not necessarily separating). Then: 

\vs{1ex}

1) $\sqrt{J} \subset J(V(J))$. \hspace{2ex} 2) If $X$ is compact, $J(V(J)) \subset \sqrt{J}$

\vs{1ex}

Thus, for compact $X$, $\sqrt{J} = J(V(J))$.
\end{proposition}
\begin{proof}
1) Let $f$ be $J$-infinitesimal and 
$x \in V(J)$. Then for each integer 
$ n \geq 0$, 
$\, nf(x) \ominus \neg f(x) = (nf \ominus \neg f)(x) = 0$. Since $[0,\,1]$ has no infinitesimals other than $0$, we have $f(x) = 0$.

\hs{4ex} 2) Since any ideal is an intersection of prime ideals \cite[1.2.14]{CDM}, it follows, from (\ref{gn1}), (\ref{gn2}) and Proposition \ref{radp1}, that we can assume $I$ to be prime. Suppose that $f$ is not a $I$-infinitesimal, and let    $n \geq 0$ be such that $nf \ominus \neg f \notin I$. From the equation $(x \ominus y) \wedge (y \ominus x) = 0$ it follows that $\neg f \ominus nf \in I$. That is, 
$\neg(n+1)f = \neg(f \oplus nf) \in I$. By (\ref{eca1}) we can take $x \in V(I)$. Then   $(\neg(n+1)f)(x) = 0$, thus $(n+1)f(x) = 1$ which implies 
$f(x) > 0$. Thus $f \notin J(V(I))$.
\end{proof}
Taking into account (\ref{chang}) above, a particular case of Proposition \ref{prenull} yields (compare with \cite[theorem 5.1]{HA}):
\begin{theorem}[Nullstellensatz] \label{null}
For any ideal $I \subset F[N]$, the ideal of term functions vanishing on the common zero locus of $I$, $V(I) \subset [0,\;1]^N$, is the infradical of $I$, that is $J(V(I)) = \sqrt{I}$. That is, if $f|_{V(I)} = 0$, then 
$\, nf \ominus \neg f \in I$ for each integer $ n \geq 0$.
\end{theorem}
Note that in other words this theorem means:

\vspace{1ex}

\emph{Given any  MV-algebra $A$ with a presentation $A = F[N]/I$, then $A$ is isomorphic to the algebra $F[N]|_{V(I)}$ of term-functions restricted to the zero-set $V(I) \subset [0,\;1]^N$, if and only if, $\sqrt{A} = \{0\}$, i.e, $A$ has no infinitesimals other than $0$.}
 
Using now that $F[N]$ is a separating subalgebra, (\ref{chang}) above, we have the following corollary of theorem \ref{null} (\cite[Th. 0.1]{CI}).
\begin{theorem} \label{nulle}
For any MV-algebra $A$ and ideal $I \subset A$, 
$Rad(I) = \sqrt{I}$, in particular, $\sqrt{A} = Rad(A)$.
\end{theorem}
\begin{proof}
 Take $N$ such that $F[N] \mr{\rho} A$ is a quotient. It suffices to prove $\rho^{-1}Rad(I) = \rho^{-1}\sqrt{I}$. We have:
$$ 
\rho^{-1}Rad(I) = Rad(\rho^{-1}I) = J(V(\rho^{-1}I)) = \sqrt{\rho^{-1}I} = \rho^{-1}\sqrt{I}.
$$
These equalities follow (in order) by Proposition 
\ref{-Rad=Rad-}, Proposition \ref{Rad=JV}, \mbox{Theorem \ref{null},} and Proposition \ref{-Irad=Irad-}.
\end{proof}
\begin{corollary}
For any MV-algebra $A$ and ideal $I \subset A$, the set of all $I$-infinitesimals is an ideal. 
\end{corollary}
\begin{corollary}
An ideal $I \subset A$ of a MV-algebra $A$ is a \emph{radical} ideal (Definition \ref{radicalideals}) if and only if $I = \sqrt{I}$.
\end{corollary}

\section{The relations between some classes of MV-algebras}

In this section we prove (except for Proposition \ref{primequasihyper}) in a syntactic elementary way, meaning first order with denumerable disjunctions, several implications (some inedited) between elementary classes of MV-algebras which in the literature are usually proved in a set theoretical semantical way.   
In the following the variables  $x, \, y, \ldots \,$ are assumed to range on some \mbox{MV-algebra $A$.}

In view of the characterisation \ref{basic} of maximal ideals we set:
\begin{definition}
An ideal $I \subset A$ of a MV-algebra $A$ is \emph{quasimaximal} 
$\iff \;\; \forall x \in A \; (x \notin I \; \iff  
\exists \, n \geq 1\; | \; \neg nx \; \in \sqrt{I})$
\end{definition}

For a MV-algebra $A$, the ideals $I$ such that the quotient algebra $A/I$ is hyperarchimedean will be called \emph{hyperradical}. Thus:
\begin{definition}
An ideal $I \subset A$ of a MV-algebra $A$ is \emph{hyperradical} if for any $x \in A$, there exists en integer $n \geq 1$ such that 
$(n+1)x \ominus nx  \in I$.
\end{definition}


For a MV-algebra $A$, the ideals $I$ such that the quotient algebra $A/I$ is quasihyperarchimedean will be called \emph{quasihyperradical}. Thus:
\begin{definition}
An ideal $I \subset A$ of a MV-algebra $A$ is \emph{quasihyperradical} if for any $x \in A$, there exists an integer $n \geq 1$ such that 
$(n+1)x \ominus nx  \in \sqrt{I}$.
\end{definition}
\begin{remark} \label{quasihyper=sqrthyper}
Clearly an ideal $I$ is quasihyperradical if and only is $\sqrt{I}$ is hyperradical.
\end{remark}

%

This illustrates a correspondence between classes of MV-algebras and notions of ideals. We have the following table:
$$
\xymatrix@R=0.1ex
  {
   simple                  \ar@{-}[r]  &   maximal            \\
   quasisimple             \ar@{-}[r]  &   quasimaximal       \\
   semisimple              \ar@{-}[r]  &   radical            \\
   chain                   \ar@{-}[r]  &   prime              \\
   hyperarchimedean        \ar@{-}[r]  &   hyperradical       \\
   quasihyperarchimedean   \ar@{-}[r]  &   quasihyperradical  
  }
$$
The next proposition is clear:
\begin{proposition} \label{quasihipersemihiper}
An ideal is hyperradical if and only if it is quasihyperradical and radical (that is, an MV-algebra is semisimple quasihyperarchimedean if and only if it is hyperarchimedean) \cqd
\end{proposition}
\begin{proposition}
Hyperradical ideals are radical ideals (that is, hyperarchimedean algebras are semisimple)
\end{proposition}
\begin{proof}
 The reader can easily check that the following holds for any ideal $I$:

\vs{-1ex}

\begin{center}
$
\; d(x \vee y, \, x) \in I  \; \iff \; y \ominus x \in I.
$
\end{center}

\vs{-1ex}

Assuming $x$ to be \mbox{$I$-infinitesimal,} it follows that for any integer 
$n \geq 1$, 
$d(\neg x \vee nx, \, \neg x) \in I$. Equivalently, 
$d(\neg(\neg x \vee nx), \, x) \in I$. But: 
$$
\neg(\neg x \vee nx) = x \wedge \neg nx = 
\neg nx \odot (nx \oplus x) = (n+1)x \ominus nx.
$$ 
Thus, 
$d((n+1)x \ominus nx , \, x)  \in I$, Take $n \geq 1$ such that $(n+1)x \ominus nx \in I$, it follows that $x \in I$, proving that $I$ is a radical ideal (compare this proof with the remark after \cite[definition 3.6.3]{COM}).
\end{proof}

\begin{proposition} \label{maximalhyper}
Maximal ideals are hyperradical ideals (that is, simple algebras are hyperarchimedean).
\end{proposition}
\begin{proof}
If $x \in I$, clearly $2x \ominus x \leq 2x \in I$. Assume $x \notin I$, and by \ref{basic} take an integer $n \geq 1$ such that $\neg nx \in I$. $nx \leq (n+1)x$, so also  $\neg (n+1)x \in I$. Then, 
$(n+1)x \ominus nx \leq d(nx,\, (n+1)x) = d(\neg nx,\, \neg (n+1)x) \in I$.
\end{proof}
\begin{proposition} \label{quasimaximalhyper}
Quasimaximal ideals are quasihyperradical ideals (that is, quasisimple algebras are quasihyperarchimedean).
\end{proposition}
\begin{proof}
The reader can check that the same proof in the previous proposition applies here.
\end{proof}
\begin{proposition} \label{long}
Prime hyperradical ideals are maximal ideals (that is, hyperarchimedean chains are simple algebras).
\end{proposition}
\begin{proof}
The reader can easily check that the following holds for any ideal $I$: 

\vs{-1ex}

\begin{center}
 $(a)$ $(x \ominus y \in I, \; y \in I \rimply x \in I)$.
\end{center}

Let $I$ be a prime hyperradical ideal, justified by \ref{basic}, it is enough to prove that if $x \notin I$, then there exist an integer $m \geq 1$ such that $\neg mx \in I$. 

Take $n$ such that $(n+1)x \odot \neg nx = (n+1)x \ominus nx \in I$. Assume (absurdum hypothesis) that $(n+1)x \odot nx \in I$. By distributivity of $\odot$ over $\vee$ it follows 
$$
(n+1)x \ominus \neg (\neg nx \vee nx) = (n+1)x \odot (\neg nx \vee nx) \in I.
$$
But
$
\neg (\neg nx \vee nx) = nx \wedge \neg nx = \neg nx \odot (nx \oplus nx) = 2nx \ominus nx.
$ 
Then, by (\ref{nimpliesmk})  $\;2nx \ominus nx \in I$. It follows by (a) above that $(n+1)x \in I$, which implies $x \in I$, contrary with our primary assumption. Thus we have 
$(n+1)x \ominus \neg nx = (n+1)x \odot nx \notin I$. Since $I$ is prime, it follows that $\neg nx \ominus (n+1)x \in I$. Finally: 
$$
\neg nx \ominus (n+1)x  = \neg nx \odot \neg (n+1)x = \neg( nx \oplus (n+1)x) = \neg (2n + 1)x.
$$
 Thus, $\neg mx \in I$ for $m = 2n + 1$.
\end{proof}
\begin{comment}   
In order to develop an elementary proof of the next two propositions it would be necessary to prove in the style of propositions \ref{quasihipersemihiper} to \ref{long} that if $I$ is a prime ideal, then $\sqrt{I}$ is maximal. 
\end{comment}
\begin{proposition} \label{primequasihyper}
Prime ideals are quasihyperradical ideals (that is, chains are quasihyperarquimedean algebras).
\end{proposition}
\begin{proof}
By Remark \ref{radicalprime} and Theorem \ref{nulle} it follows that if $I$ is a prime ideal, $\sqrt{I}$ is maximal, thus by \ref{maximalhyper} it is hyperradical. Then, Remark \ref{quasihyper=sqrthyper} finishes the proof.
\end{proof}

\begin{proposition}
Prime radical ideals are maximal ideals (that is, semisimple chains are simple algebras).
\end{proposition}
\begin{proof}
By proposition \ref{primequasihyper} the ideal is  quasihyperradical and radical, thus by \ref{quasihipersemihiper} it is hyperradical. The proof finishes by proposition \ref{long}.
\end{proof}


\begin{thebibliography}{99}



\bibitem{CI} Cignoli R. \textsl{private communication} (2011).

\bibitem{COM} Cignoli R., D'Ottaviano I., Mundici
  D., Algebraic Foundations of Many-valued Reasoning, \textsl{Trends
  in Logic Vol 7, Kluwer Academic Publishers} (2000).

\bibitem{CDM} Cignoli R., Dubuc E. J., Mundici
  D., Extending Stone duality to multisets and locally finite
 MV-algebras, \textsl{Journal of Pure and Applied Algebra, 189, 37-59} (2004).



\bibitem{DP1} Dubuc E. J., Poveda Y. Representation theory of MV-algebras, \textsl{Ann. Pure Appl. Logic, 161, 1024-1046} (2010).



 

\bibitem{HA} Harris J. Algebraic Geometry, \textsl{Graduate Texts in Mathematics 133, Springer} (1992). 













\end{thebibliography}
\end{document}